\documentclass[11pt]{amsart}
\pdfoutput=1
\usepackage{amssymb}
\usepackage{amsmath}
\usepackage{amsfonts}

\usepackage{graphicx}

\makeatletter 
 
 \@addtoreset{equation}{section}
\makeatother

\newtheorem{theorem}{Theorem}[section]

\newtheorem{claim}[theorem]{Claim}
\newtheorem{prop}[theorem]{Proposition}

\theoremstyle{definition}
\newtheorem{definition}[theorem]{Definition}
\newtheorem{remark}[theorem]{Remark}

\newcommand{\Z}{\mathbb{Z}}
\newcommand{\RR}{\mathbb{R}}
 
\newcommand{\K}{\mathcal{K}}
\newcommand{\LL}{\mathcal{L}}

\newcommand{\s}{\mathcal{S}}

\newcommand{\ol}{\overline}
\newcommand{\ee}{\stackrel{e}{=}}

\begin{document}
\author{Keiko Kawamuro}
\address{Department of Mathematics, Rice University, Houston, TX 77005}
\email{kk6@rice.edu}

\title{Connect sum and transversely non simple knots}


\begin{abstract} We prove that transversal non-simplicity is preserved under taking connect sum, generalizing Ve\'rtesi's result \cite{V}.
\end{abstract}
\maketitle
\section{Introduction}

The goal of this paper is to prove:

\begin{theorem}\label{main-thm}  
Let $K_1, K_2$ be prime knot types in $S^3$. Let $T_1, T_1'$ $($resp. $T_2, T_2')$ be transverse knots in $(S^3, \xi_{sym})$ of topological type $K_1$ $($resp. $K_2)$. Suppose that
\begin{enumerate}
\item $T_1, T_1'$ have the same self linking number but are not transversely isotopic, and 
\item $T_2, T_2'$ are transversely isotopic and cannot be transversely destabilized. 
\end{enumerate}
Then the connect sums $T_1 \# T_2$ and $T_1' \# T_2'$ are not transversely isotopic. 

{\em (}We allow the possibility that $K_1, K_2$ have the same topological type and $T_1, T_2$ are transversely isotpic.{\em )}
\end{theorem}

\begin{remark}
If $T_2$ is transversely destabilizable and $T_1, T_1'$ are related to each other by a negative flype move (see \cite{BW} for definition), then $T_1 \# T_2$ is transversely isotopic to $T_1' \# T_2'$.
\end{remark}

The idea behind Theorem~\ref{main-thm} was a result of Ve\'rtesi, who proved a specialized version of it in her paper \cite{V}. Her result holds only when the transversally non-simple knots in question can be distinguished by invariants in Heegaard Floer homology theory studied in \cite{NOT}. We make no such restrictions. In fact, Birman-Menasco proved the existence of  infinitely many transversely non simple knots \cite{BM} that the Heegaard Floer homology invariants do not distinguish \cite{NOT}.
 
We will give two proofs. The first proof is given in Section~\ref{sec2}. It uses the theory of transversal closed braids, and is based upon ideas in \cite{Be}, \cite{BM-composite}, \cite{BW}, \cite{OS}, \cite{W}. The second proof is given in Section~\ref{sec3}. It is inspired by a suggestion of John Etnyre that our theorem ought to follow from a theorem of Etnyre-Honda \cite{EH}, and uses techniques based upon the well-known idea that every transversal knot type can be represented by a transversal pushoff of some Legendrian knot.

\textbf{Acknowledgments } 
The author would like to thank John Etnyre for helpful comments and for pointing out that Theorem~\ref{main-thm} follows from a theorem of Etnyre and Honda. She is also grateful to the referee, who read the original manuscript carefully and suggested improvement on exposition.

\section{Proof of Theorem~\ref{main-thm}}\label{sec2} 
Throughout this paper, $T, T_{i=1,2}$ denote transversal knots in $(S^3, \xi_{sym})$ the symmetric contact structure of $S^3$. Regard $S^3$ as a one point compactification of $\RR^3$ equipped with the cylindrical coordinates $(r, \theta, z)$.  Thanks to Bennequin \cite{Be} we identify transversal knots in $(S^3, \xi_{sym})$ with closed braids in $\RR^3$ about the $z$-axis.

\begin{definition}
Suppose $T_1, T_2$ have braid presentations 
$$T_1 = b\ \sigma_{n-1} b'\ \sigma_{n-1}^{-1}, \quad T_1 = b\ \sigma_{n-1}^{-1} b'\ \sigma_{n-1}$$ 
where $n$ is the braid index and $b, b'$ are some braid words in $\sigma_1, \cdots, \sigma_{n-2},$ the standard generators of the braid group $B_n$.
See Figure~\ref{ex-move}. Then we say $T_1$ and $T_2$ are related to each other by an {\em exchange move}.
\begin{figure}[htpb!]
\begin{center}
\begin{picture}(276, 65)
\put(0,0){\includegraphics{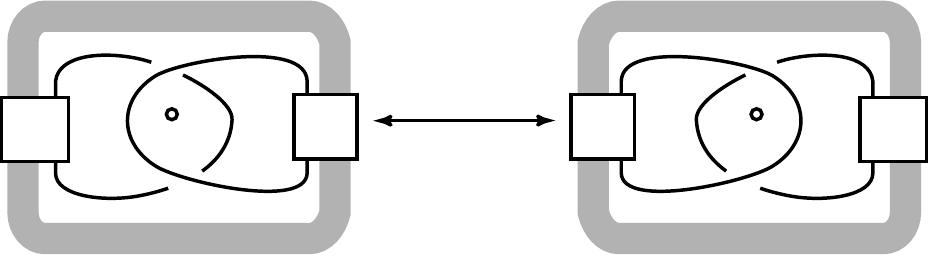}}
\put(5, 33){$b$} \put(89, 33){$b'$} \put(170, 33){$b$} \put(254, 33){$b'$} \put(40, 30){\tiny{ $z$-axis}}
\end{picture}
\caption{An exchange move between $T_1$ (left) and $T_2$ (right). Thick gray bands are $(n-1)$ parallel strands and $b, b'$ are some braidings.}\label{ex-move}
\end{center}
\end{figure}
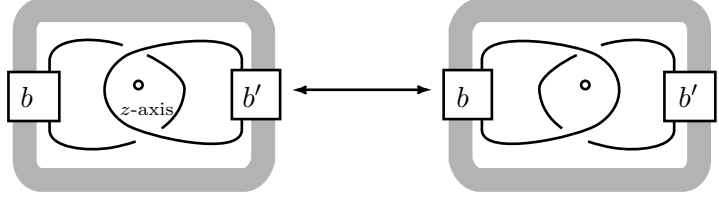
\end{definition}

As shown in \cite[Lemma 1]{BW}, an exchange move is a composition of a positive braid stabilization and a positive braid destabilization. Thus, an exchange move is a transversal isotopy. We use the following notations:
\begin{itemize}
\item $T_1 = T_2$ if  $T_1, T_2$ are braid isotopic (conjugate);
\item $T_1 \stackrel{e}{=} T_2$ if  $T_1, T_2$ are exchange equivalent;
\item $T_1 \sim T_2$ if $T_1, T_2$ are transversely isotopic;
\item $S^+(T)$ for a transverse knot obtained by a number of positive braid stabilizations of $T$. It is known that $S^+(T) \sim T$.
\end{itemize}
Notice that $T_1 = T_2 \Rightarrow T_1 \stackrel{e}{=} T_2 \Rightarrow T_1 \sim T_2$.

\begin{definition}\label{connect sum def}  
We define the  {\em braid connect sum} $T_1 \# T_2$ of $T_1$ and $T_2$ as in Figure~\ref{connect-sum}. 
\end{definition}
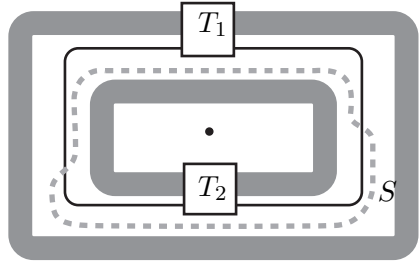
\begin{figure}[htpb!]
\begin{center}
\begin{picture}(165, 100)
\put(0,0){\includegraphics{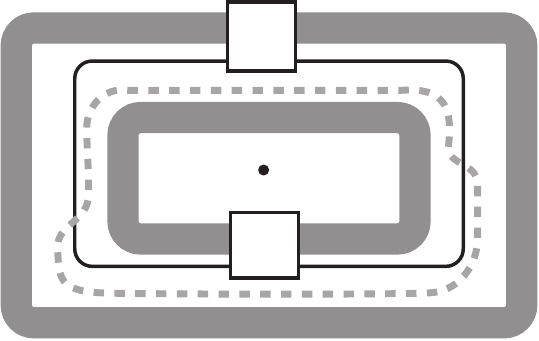}}
\put(72, 87){$T_1$}  \put(72, 25){$T_2$} \put(140, 23){$S$}
\end{picture}
\caption{Connect sum $T_1 \# T_2$ and dividing sphere $S$ (dashed). Thick bands are multi strands.}\label{connect-sum}
\end{center}
\end{figure}

This definition of connect sum is well defined thanks to Birman-Menasco \cite{BM-composite}:
\begin{theorem}\label{BM-sum} \cite[Composite braid theorem]{BM-composite}
Any composite braid can be reduced to the form in Figure~\ref{connect-sum} by exchange moves and braid isotopy.  
\end{theorem}

\begin{definition} 
A {\em transverse  stabilization} of $T$ is a negative braid stabilization, that is, addition of an negative trivial kink about the $z$-axis. We call the inverse operation  {\em transverse destabilization}.
\end{definition} 

Now we are ready to prove our main theorem.

{\em Proof of Theorem~\ref{main-thm} } 
Suppose, on the contrary, that $T_1 \# T_2 \sim T_1' \# T_2'$. Thanks to Orevkov-Shevchishin \cite{OS} and Wrinkle \cite{W}, after a number of positive braid stabilizations we get $S^+(T_1 \# T_2) = S^+(T_1' \# T_2').$ Due to Birman-Wrinkle \cite[Lemma 2]{BW}, one can slide a trivial stabilization loop to any place around the braid by exchange moves and braid isotopy. Therefore,
\begin{equation}\label{exchange-moves}
T_1 \# S^+(T_2)  \ee S^+(T_1 \# T_2) = S^+(T_1' \# T_2') \ee T_1' \# S^+(T_2').
\end{equation}
To simplify notation, since $S^+(T)\sim T$, we will denote $S^+(T_2)$ by $T_2$ and $S^+(T_2')$ by $T_2'$. Let $f:S^3 \to S^3$ be a diffeomorphism corresponding to the composition of the exchange moves and braid isotopy in (\ref{exchange-moves}) so that 
$$f(T_1 \# T_2)=T_1' \# T_2'.$$
We may think that the restriction of $f$ to $T_1$ does not change $\theta$-coordinate i.e., 
\begin{equation}\label{theta}
f|_{T_1}(r,\theta, z)=(r', \theta, z').
\end{equation}

In the following, we will deduce $T_1 \sim T_1'$, which contradicts our assumption.

Let $S \subset S^3$ (resp. $S'$) be a $2$-sphere separating $T_2$ (resp. $T_2'$) and $T_1$ (resp. $T_1'$) as in Figure~\ref{connect-sum}. Let $p, q$ (resp. $p', q'$) denote the intersection points of $S \cap T_1 \# T_2$ (resp. $S' \cap T_1' \# T_2'$). Let $\tilde{T_1}, \tilde{T_2}$ (resp. $\tilde{T_1'}, \tilde{T_2'}$) be two arcs obtained by cutting $T_1 \# T_2$ (resp. $T_1' \# T_2'$) at $p$ and $q$ (resp. $p', q'$). Suppose $\partial T_2=\{-p\} \cup \{q \}$ (resp. $\partial T_2'=\{-p'\} \cup \{q' \}$) with respect to the positive orientation of the braid $T_1 \# T_2$. We have
$$T_1' \# T_2' = \tilde{T_1'}\cup \tilde{T_2'}= f(\tilde{T_1} \cup \tilde{T_2}) = f(\tilde{T_1}) \cup f(\tilde{T_2}).$$   
Let 
$$A= f( \tilde{T_2}) \cap \tilde{T_1'},   \quad B=f( \tilde{T_2}) \cap  \tilde{T_2'},   \quad C=\tilde{T_2'} \setminus B, \quad D=\tilde{T_1'} \setminus A,$$ 
where arcs $A, B, C$ or $D$ may be empty and may have more than one component. 

By small perturbation, $f(S) \cap S'$ consists of a number of disjoint  circles and $f(S) \cup S'$ divides $S^3$ into a number of $3$-balls. Since $T_1' \# T_2'$ intersects $f(S)$ (resp. $S'$) only at two points $f(p), f(q)$ (resp. $p', q'$), some of the balls do not intersect $T_1' \# T_2'$. We deform $f(S)$ to remove such empty balls, starting with the innermost one without moving $T_1' \# T_2'$. See Figure~\ref{modify}. We use the same notation $f(S)$ for the changed $f(S)$. 
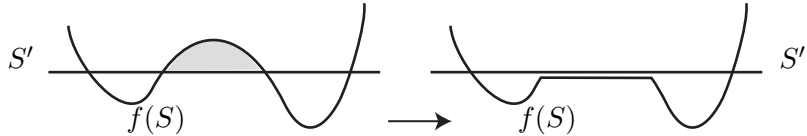
\begin{figure}[htpb!]
\begin{center}
\begin{picture}(276, 55)
\put(0,0){\includegraphics{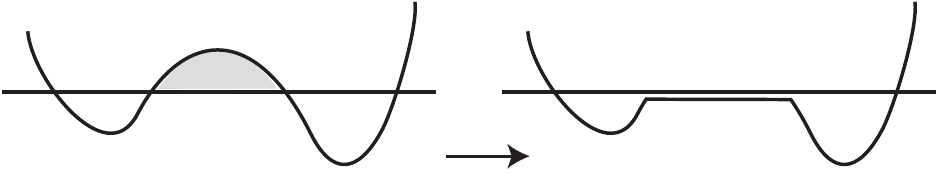}}
\put(-15, 25){$S'$} \put(30, 3){$f(S)$} \put(277, 25){$S'$} \put(178, 3){$f(S)$}
\end{picture}
\caption{Removing the shaded innermost empty $3$-ball.  }\label{modify}
\end{center}
\end{figure}

Eventually we have the following five cases. Note that by (\ref{theta}), we have $\theta_{f(p)}=\theta_p$ and $\theta_{f(q)}=\theta_ q.$ See Figure~\ref{1.pdf}. 
\begin{itemize}
\item $f(S) \cap S'$ consists of one circle and $S^3$ is divided by $f(S) \cup S'$ into four $3$-balls. Cyclic order of $\theta_p, \theta_q, \theta_{p'}, \theta_{q'}$ for each case is;
$$\mbox{Case (1) } \theta_p < \theta_{q'} < \theta_q < \theta_{p'}; \quad  \mbox{Case (1$'$) } \theta_p < \theta_{p'} < \theta_q < \theta_{q'}.$$
\item $f(S)$ and $S'$ are disjoint, and $S^3$ is divided by $f(S) \cup S'$ into two $3$-balls and $S^2 \times (0,1)$.
$$\mbox{Case (2) } \theta_p < \theta_{q'} < \theta_{p'} < \theta_q; \quad \mbox{Case (3) } \theta_p < \theta_{p'} < \theta_{q'} < \theta_q; $$
$$\mbox{Case (4) } \theta_p < \theta_q < \theta_{q'} < \theta_{p'}; \quad \mbox{Case (5) }\theta_p < \theta_q < \theta_{p'} < \theta_{q'}.$$
\end{itemize}
\begin{figure}[htpb!] 
\begin{center}
\begin{picture}(296, 423)
\put(-6,0){\includegraphics{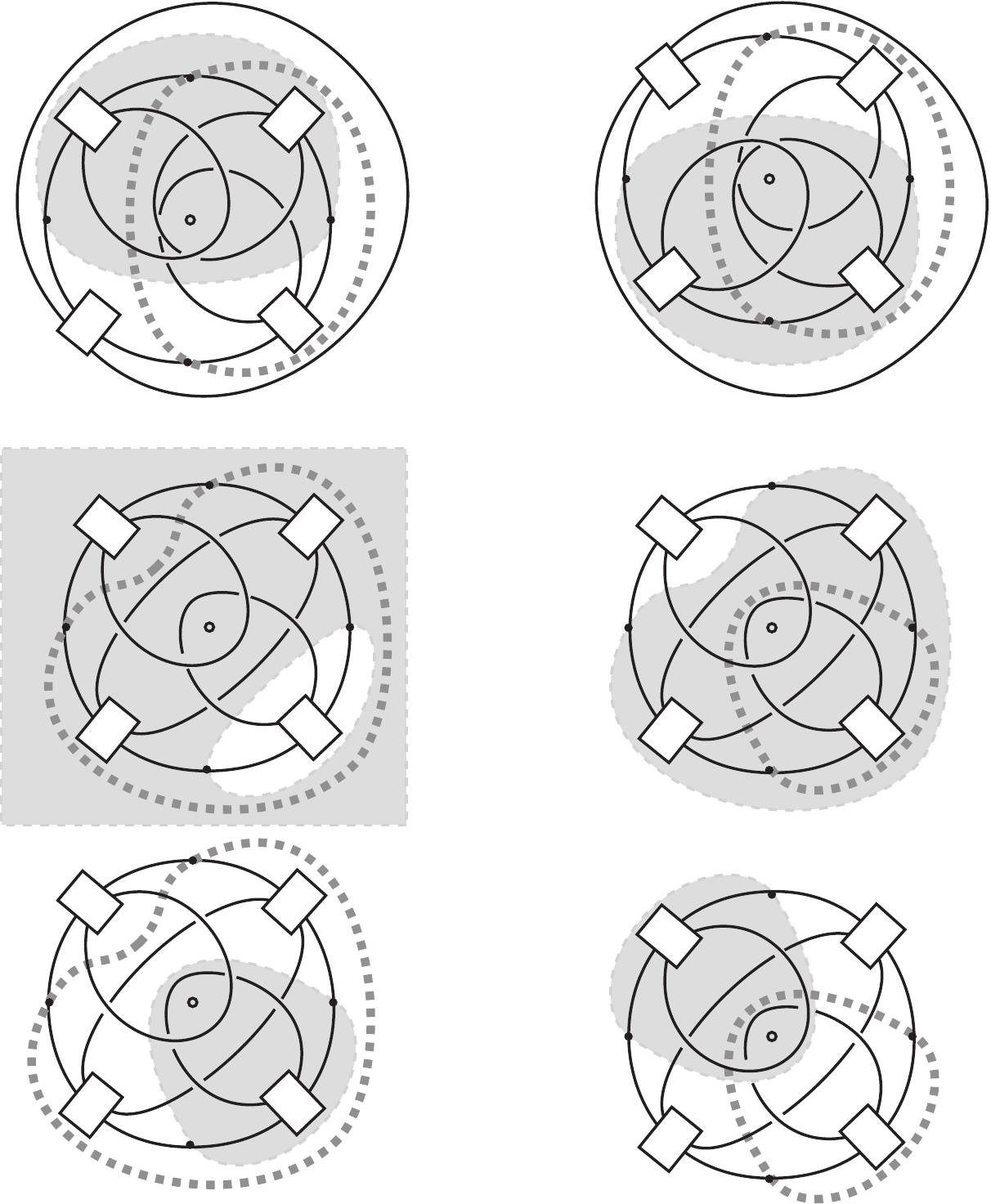}}
\put(21, 380){$A$} \put(93, 380){$B$} \put(93, 308){$C$} \put(20, 308){$D$}  
\put(225, 395){$D$} \put(297, 395){$C$} \put(225, 325){$A$} \put(297, 323){$B$} 
\put(28, 236){$A$} \put(98, 238){$B_1$} \put(27, 166){$B_2$} \put(100, 164){$C$} 
\put(297, 239){$A_1$} \put(224, 165){$A_2$} \put(297, 165){$B$} \put(225, 237){$D$} 
\put(94, 105){$C_1$} \put(22, 33){$C_2$} \put(94, 33){$B$} \put(20, 105){$D$} 
\put(295, 91){$D_1$} \put(223, 19){$D_2$} \put(297, 20){$C$} \put(225, 92){$A$} 
\put(115, 344){$f(p)$} \put(-20, 344){$f(q)$} \put(60, 404){$q'$} \put(60, 304){$p'$}  
\put(190, 360){$f(p)$} \put(260, 300){$p'$} \put(317, 360){$f(q)$} \put(260, 400){$q'$}  
\put(118, 198){$f(p)$} \put(5, 202){$p'$} \put(65, 260){$q'$} \put(65, 160){$f(q)$} 
\put(319, 207){$q'$} \put(195, 202){$f(p)$} \put(264, 260){$f(q)$} \put(264, 160){$p'$} 
\put(113, 70){$f(q)$} \put(0, 70){$p'$} \put(61, 128){$q'$} \put(61, 28){$f(p)$} 
\put(320, 65){$q'$} \put(194, 57){$f(q)$} \put(264, 115){$f(p)$} \put(264, 15){$p'$} 
\put(-8, 400){(1)} \put(196, 400){(1$^\prime$)} \put(4, 256){(2)} \put(199, 256){(3)} \put(-4, 108){(4)} \put(199, 108){(5)}
\end{picture}
\caption{Transverse knot $T_1' \# T_2'$ where $A_1 \cup A_2 = A$, $B_1 \cup B_2 = B$, $C_1 \cup C_2 = C$ and $D_1 \cup D_2 = D$. Sphere $S^\prime=\partial \Delta'$ is dashed. Sphere $f(S)$ and the $3$-ball $f(\Delta)$ containing $f(\tilde{T_2})$ is shaded. Braid strands may be weighted.}\label{1.pdf}
\end{center}
\end{figure}
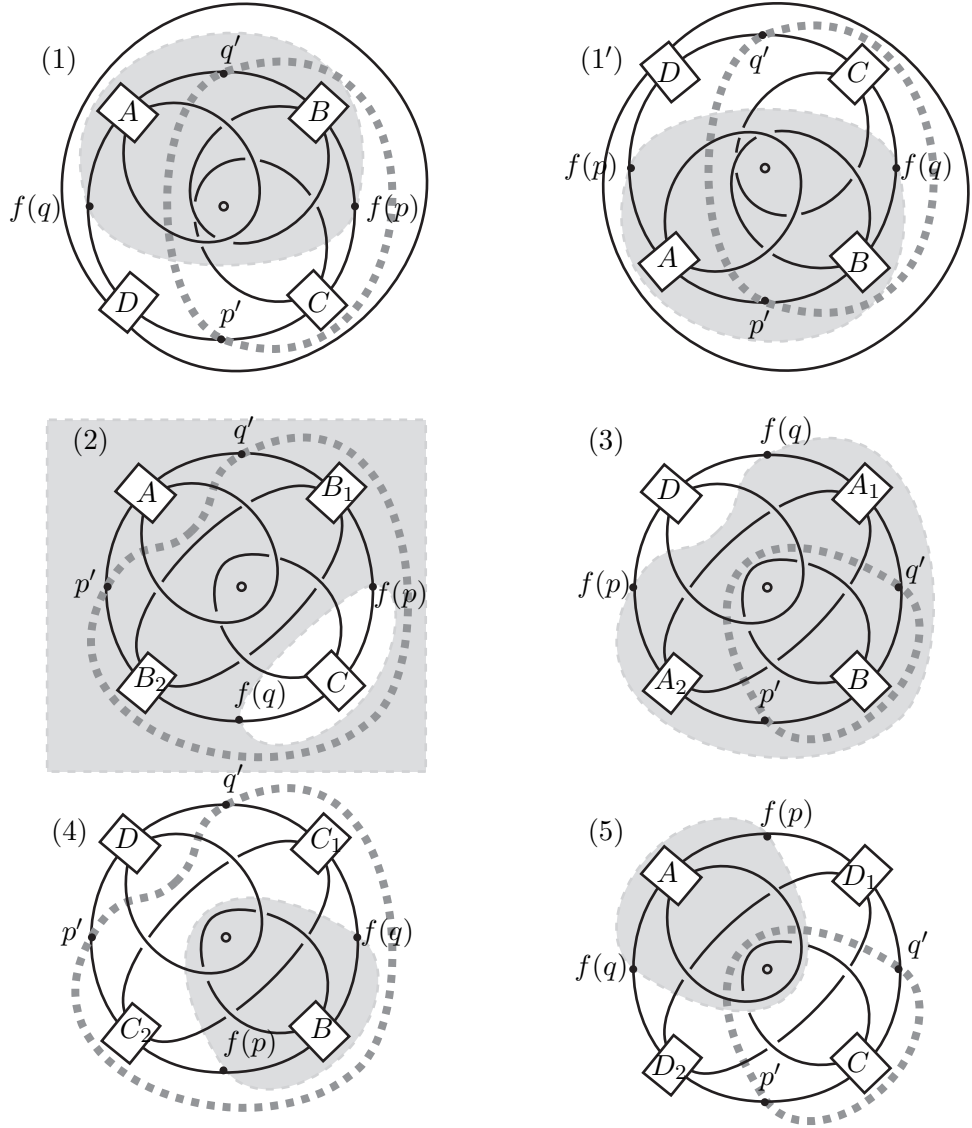

Let $\ol{A \cup D}$ $($resp. $\ol{B \cup C})$ be a closed braid obtained by filling the braid blocks $B, C$ $($resp. $A, D$) with trivial braid strands of braid index $1$. They are determined uniquely up to braid isotopy.

\begin{claim}\label{claim T_1'} 
We have $\ol{A \cup D}=T_1'$ $($resp. $\ol{B \cup C}=T_2')$. 
\end{claim}

\begin{proof}
This is clear from Definition~\ref{connect sum def} of the connect sum.
\end{proof}

Let $\ol{C\cup D}$ $($resp. $\ol{A\cup B})$ be a closed braid obtained by filling the braid blocks $A, B$ $($resp. $C, D$) with trivial braid strands of braid index $1$. Note that $\ol{C \cup D}$ $($resp. $\ol{A\cup B})$ is unique up to exchange moves, since after an exchange move $f$ the sphere $f(S)$ is pierced by the braid axes more than twice in general and there may be several ways to take the braid closure. 

\begin{claim}\label{claim T_1} 
We can extend $f(\tilde{T_1})=C\cup D$ $($resp. $f(\tilde{T_2})=A\cup B)$ to the closed braid $\ol{C\cup D}$ $($resp. $\ol{A\cup B})$ and the extension satisfies $\ol{C\cup D}\ee T_1$ $($resp. $\ol{A\cup B} \ee T_2)$.
\end{claim}
%
\begin{proof}
Let $\Delta \subset S^3 \setminus S$ be the $3$-ball containing $\tilde{T_2}$. Join the end points $p, q$ of $\tilde{T_1}$ by an arc $\alpha \subset \Delta$ so that $T_1 = \tilde{T_1} \cup \alpha$. 

Suppose $f=f_k \circ \cdots \circ f_0$ where $f_i$ ($i=0,1,\cdots, k$) is an exchange move. We may assume, if necessary, as in Figure~\ref{ex-move-2} by using some braid isotopy with property (\ref{theta}), that $p, q$ are fixed by $f_i$. That is, each of the two exchange arcs is a sub-arc of either $\tilde{T_1}$ or $\tilde{T_2}$.
\begin{figure}[htpb!]
\begin{center}
\begin{picture}(265, 60)
\put(0,0){\includegraphics{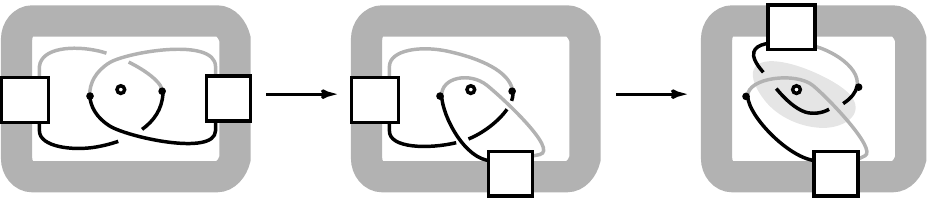}}
\put(5, 25){$b$} \put(62, 25){$b'$} \put(105, 25){$b$} \put(143, 3){$b'$} \put(238, 3){$b'$} \put(226, 45){$b$} \put(19, 26){$p$} \put(48, 28){$q$} \put(120, 26){$p$} \put(150, 28){$q$}  \put(210, 20){$p$} \put(250, 28){$q$}
\end{picture}
\caption{By braid isotopy, points $p$ and $q$ can be outside the shaded exchange domain. Arc $\tilde{T_1}$ is colored gray and $\tilde{T_2}$ is black.}
\label{ex-move-2}
\end{center}
\end{figure}

Here we recall some of Birman-Menasco's foundational work in \cite{BM-composite}. Let $H_{\theta_0} \subset \RR^3$ be the half-plane $\{ (r, \theta_0, z) | 0< r, z \in \RR \}$. For all but a finite number of $\theta \in [0, 2\pi)$ the intersection $S \cap H_\theta$ is a disjoint union of simple closed curves and properly embedded arcs, in which case $H_\theta$ is called {\em non-singular}. Thanks to \cite[Lemma 1 and p.135]{BM-composite} we may assume that there are no simple closed curves. When $H_\theta$ is non-singular, we call an arc $\beta \subset S \cap H_\theta$ {\em essential} if the both components of $H_\theta$ split along $\beta$ are pierced by our transverse knot.

An exchange move of a composite braid with separating sphere $S$ is done by three steps. See Figure~\ref{room}. First,  without moving the braid we change the shape of $S$ to make a ``room'' for the coming exchange move, which can be done in the exchange domain (the shaded $3$-ball in the right sketch of Figure~\ref{ex-move-2}) away from $p$ and $q$. Second, we move the braid by fixing $S$. Third, move $S$ by isotopy in order to remove all the inessential arcs from the inner-most one (as in \cite[p.120]{BM-composite}) if they occur in the above procedure. 
\begin{figure}[htpb!]
\begin{center}
\begin{picture}(333, 191)
\put(0,0){\includegraphics{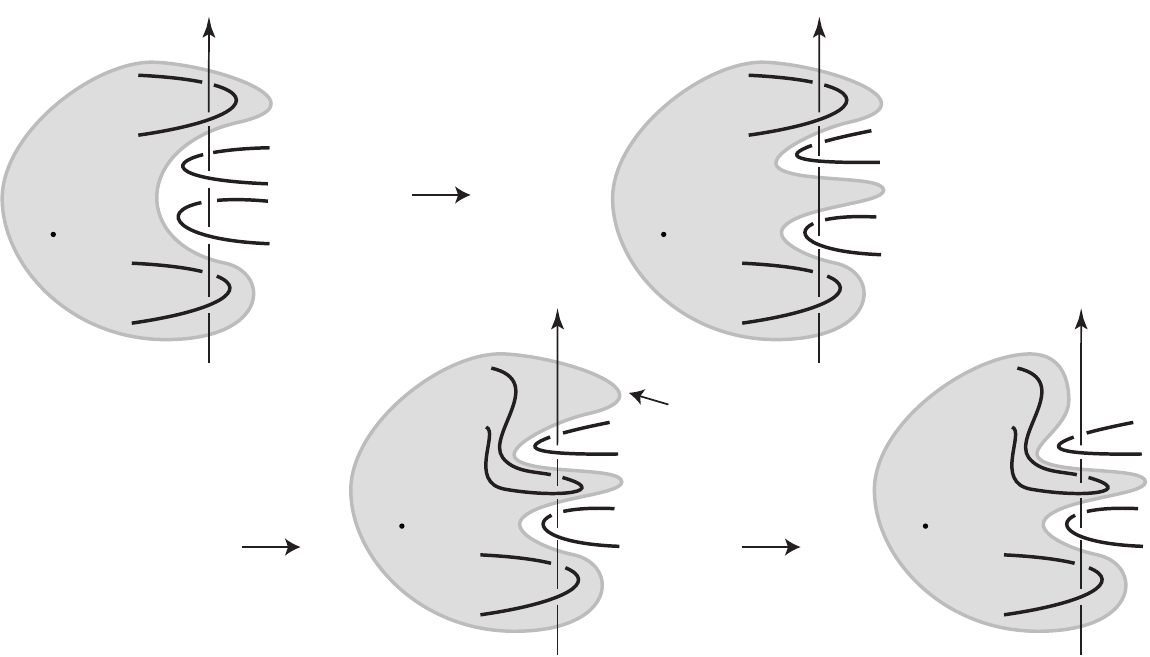}}
\put(120, 140){(1)} \put(195, 69){{\scriptsize inessential}}  \put(195, 60){{\scriptsize arc}} \put(70, 36){(2)} \put(216, 36){(3)} \put(16, 125){$q$} \put(192, 125){$q$} \put(120, 36){$q$} \put(270, 36){$q$} 
\end{picture}
\caption{The three steps of an exchange move. The $3$-ball $\Delta$ is shaded. Points $p$ (not in the sketch) and $q$ are fixed.}
\label{room}
\end{center}
\end{figure}
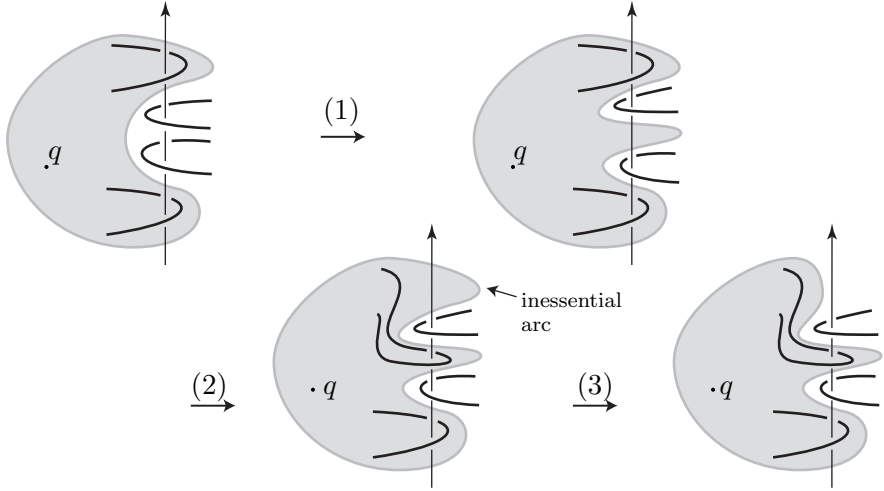

Based on this, for each exchange move $f_i$ we define how the joining arc $\alpha$ changes. First, change $S$ as $f_i$ does by fixing $\tilde{T_1}$ but moving $\alpha$ in the $3$-ball $\Delta$ by some exchange move if necessary, and $p, q \in S$ are fixed. Second, (a) if an exchange move $f_i$ involves sub-arcs of $\tilde{T_1}$ then change $\tilde{T_1} \cup \alpha$ to $f_i(\tilde{T_1}) \cup \alpha,$  (b) otherwise, do not change $\tilde{T_1} \cup \alpha$ at all. Thus, $\tilde{T_1} \cup \alpha$ and its result after the changes are related to each other by an exchange move up to braid isotopy. Third, remove the inessential arcs as $f_i$ does. Since the arc $\alpha$ is of braid index $=1$, even after the third step there may exist inessential arcs. 

Repeating this construction for each $f_i$, we obtain a closed braid $\ol{f(\tilde{T_1})}$ extending $f(\tilde{T_1})=C \cup D$ and $\ol{f(\tilde{T_1})}\ee T_1$.

Furthermore, $\ol{f(\tilde{T_1})}$ is exchange equivalent to the closed braid $A \cup B \cup C \cup D$ filling the braid blocks $A, B$ (or $A_i, B_j,$ $i, j=1,2$) by trivial arcs of braid index $=1$, i.e., $\ol{f(\tilde{T_1})}\ee \ol{C \cup D}$. 
\end{proof}

We continue the proof of Theorem~\ref{main-thm}.

Construct closure $\ol{A}$ by filling the braid boxes $B, C, D$ by trivial braid arcs of braid index $=1$. Closure $\ol{A}$ is unique up to exchange move. Similarly, construct closures $\ol{B}$, $\ol{C}$, $\ol{D}$. By this construction and Theorem~\ref{BM-sum}, in all the five cases we have $T_1' \# T_2' = A \cup B \cup C \cup D = \ol{A} \# \ol{B} \# \ol{C} \# \ol{D}$.

By Claims~\ref{claim T_1'}, \ref{claim T_1} and Theorem~\ref{BM-sum} we have
\begin{eqnarray*}
T_1 &\ee &  \ol{C \cup D} \ee \ol{C} \# \ol{D}, \\
T_2 & \ee &  \ol{A \cup B} \ee \ol{A} \# \ol{B}, \\
T_1' & = &  \ol{A \cup D} \ee \ol{A} \# \ol{D}, \\
T_2' & = &  \ol{B \cup C} \ee \ol{B} \# \ol{C}. 
\end{eqnarray*}

(Case 1) Recall $K_2$ is the topological type of $T_2 \sim T_2'$ and it is a prime knot. Since $T_2 \ee \ol{A} \# \ol{B}$ and $T_2' \ee \ol{B} \# \ol{C}$  we have two cases to study.

(Case 1.1)
Suppose that topologically $\ol{B}$ is $K_2$ and $\ol{A}, \ol{C}$ are the unknot. Since $T_2$ and $T_2'$ are not transversely destabilizable and the unknot is exchange reducible \cite[Theorem 1]{BM-unknot}, it follows that $\ol{A}, \ol{C}$ are transversely isotopic to the $1$-strand braid representative of the unknot. Thus 
$$T_1 \ee \ol{C} \# \ol{D} \sim \ol{D} \sim \ol{A} \# \ol{D} \ee T_1'.$$

(Case 1.2) Suppose that topologically $\ol{B}$ is the unknot and $\ol{A}, \ol{C}$ are $K_2$. Since $T_2 \sim T_2'$ cannot be transversely destabilized, $\ol{B}$ is transversely isotopic to the $1$-strand braid and $\ol{A} \sim \ol{A} \# \ol{B} \ee T_2 \sim T_2' \ee \ol{B} \# \ol{C} \sim \ol{C}$. Thus
$$T_1 \ee \ol{C} \# \ol{D} \sim \ol{A} \# \ol{D} \ee T_1'.$$

Similar arguments hold for (Case 1$'$).

(Case 2) Since $K_2$ is a prime knot, we have two cases to study. 

(Case 2.1) Suppose that topologically $\ol{B}$ is $K_2$ and $\ol{A}, \ol{C}$ are the unknot. Since $T_2$  cannot be transversely destabilized, $\ol{A}, \ol{C}$ are transversely isotopic to the $1$-strand braid. Thus 
$$T_1 \ee \ol{C} \sim \ol{A} = T_1'.$$

(Case 2.2) Suppose that topologically $\ol{B}$ is the unknot and $\ol{C}$ is $K_2$. Since $T_2'\ee \ol{B} \# \ol{C}$  cannot be transversely destabilized, $\ol{B}$ is transversely isotopic to the $1$-strand braid. Therefore,
$$T_1 \ee \ol{C} \sim \ol{C} \# \ol{B} \ee T_2' \sim T_2 \ee \ol{A} \# \ol{B} \sim \ol{A} = T_1'.$$

(Case 3) Since $\ol{A} \# \ol{B} \ee T_2 \sim T_2' = \ol{B}$  cannot be transversely destabilized, $\ol{A}$ is transversely isotopic to the $1$-strand braid. Therefore, 
$$T_1 \ee \ol{D} \sim \ol{D} \# \ol{A} \ee T_1'.$$

(Case 4) Since $\ol{B} \# \ol{C} \ee T_2' \sim T_2 \ee \ol{B}$  cannot be transversely destabilized, $\ol{C}$ is transversely isotopic to the $1$-strand braid. Therefore,
$$T_1 \ee \ol{C} \# \ol{D} \sim \ol{D} = T_1'.$$

(Case 5) We have $\ol{A} \ee T_2 \sim T_2' = \ol{C}$ thus $T_1 \ee \ol{C} \# \ol{D} \ee \ol{A} \# \ol{D} \ee T_1'.$

In all the cases, we obtain $T_1 \sim T_1'$ which contradicts our assumption that $T_1 \nsim T_1'$.
\hfill $\Box$

\section{Appendix}\label{sec3} 

In this section, we give an alternative proof of Theorem~\ref{main-thm} as a corollary of Etnyre-Honda's classification of connected sum Legendrian knots \cite[Theorem 3.4]{EH}. Since our ambient manifold is $S^3$ we use its $\RR^3$-version taken from \cite{E}.

Let $\K \subset \RR^3$ be a topological knot type and $\LL(\K)$ be the set of Legendrian representatives of $\K$. We denote by $\s_\pm(L)$ the $\pm$-stabilization of the Legendrian knot $L$.
\begin{theorem}\label{EH-thm}
\cite[Theorem 3.4]{EH}  \cite[Theorem 5.11]{E}
Let $\K = \K_1 \# \cdots \# \K_n$ be a topological connected sum knot type in $\RR^3$. The map 
$$\frac{\LL(\K_1) \times \dots \times \LL(\K_n)}{\approx} \to \LL(\K_1 \# \cdots \# \K_n)$$
is a bijection where the equivalence relation $\approx$ is generated by
\begin{enumerate}
\item $( \dots, \s_{\pm}(L_i), \dots , L_j, \dots ) \approx ( \dots, L_i,  \dots , \s_{\pm}(L_j), \dots )$
\item $(L_1, \dots, L_n) \approx (L_{\sigma(1)}, \dots, L_{\sigma(n)})$  where $\sigma$ is a permutation of $1, \dots , n$ such that $\K_i = \K_{\sigma(i)}. $
\end{enumerate}
\end{theorem}
We also recall a theorem by Epstein-Fuchs-Meyer \cite{EFM}: Let $L \subset (S^3, \xi_{std})$ be a Legendrian knot and $T_{\pm}(L)$ be its positive and negative transverse push offs. 
\begin{theorem}\label{thm-EFM}
\cite[Theorem 2.1]{EFM}
Legendrian knots $L_1, L_2 \subset (S^3, \xi_{std})$ are negatively stably isotopic $($i.e., $\s_-^k(L_1) = \s_-^l(L_2)$ for some $k, l \geq 0)$ if and only if $T_{\pm}(L_1) \sim T_{\pm}(L_2)$.
\end{theorem}

The next proposition explains the relationship between positive Legendrian stabilization and transverse stabilization: Let $S(T)$ be a transverse stabilization of $T$. Under the identification of $T$ with a closed braid, $S(T)$ is an negative braid stabilization of $T$.
\begin{prop}\label{prop}
Let $L \subset (S^3, \xi_{std})$ be a Legendrian knot. We have $T_+(\s_+(L))\sim S(T_+(L))$.
\end{prop}

Here is an alternative proof of Theorem~\ref{main-thm}.
\begin{proof}
Suppose that $T_1 \# T_2 \sim T_1' \# T_2'$. Let $L_1 \# L_2$ (resp. $L_1' \# L_2'$) be a Legendrian push-off of $T_1 \# T_2$ (resp. $T_1' \# T_2'$) so that
$$T_+(L_1 \# L_2) \sim T_1 \# T_2 \sim T_1' \# T_2' \sim T_+(L_1' \# L_2').$$
By Theorem~\ref{thm-EFM}, $\s_-^k(L_1 \# L_2) = \s_-^l(L_1' \# L_2')$ in $\LL(\K_1 \# \K_2)$ for some $k, l \geq 0$. Since Legendrian stabilization is well defined (we can move the zig-zags anywhere) we have
$L_1 \# \s_-^k(L_2) = L_1' \# \s_-^l(L_2')$.
By Theorem~\ref{EH-thm}, 
$$(L_1, \s_-^k(L_2))\approx (L_1',  \s_-^l(L_2')).$$ 
Recall our assumption that $T_2 \sim T_2'$ cannot be tansversely destabilized. Thus Proposition~\ref{prop} implies that: 
\begin{claim}\label{L_2}
$L_2$ and $L_2'$ cannot be positive Legendrian destabilizable. 
\end{claim}

Suppose that $\K_1 \neq \K_2$. By the definition of $\approx$ in Theorem~\ref{EH-thm} and Claim~\ref{L_2}, we have for some $m, n \in \Z$ and $x, y \in \Z_{\geq 0}$. 
$$(\s_+^{-x} \s_-^m(L_1), \ \s_+^x \s_-^{k-m}(L_2))= (\s_+^{-y} \s_-^n(L_1'), \   \s_+^y \s_-^{l-n}(L_2')).$$ 
By Theorem~\ref{thm-EFM} and Proposition~\ref{prop} we have
$$S^x(T_2) \sim T_+(\s_+^x \s_-^{k-m}(L_2)) = T_+(\s_+^y \s_-^{l-n}(L_2')) \sim S^y(T_2')$$
and obtain $x=y$. Then, $$S^{-x}(T_1) \sim T_+(\s_+^{-x} \s_-^m(L_1)) \sim T_+(\s_+^{-y} \s_-^n(L_1')) \sim S^{-x}(T_1')$$
and we obtain $T_1 \sim T_1'$, which is a contradiction. 

Suppose that $\K_1 = \K_2$. We have either
$$(\s_+^{-x} \s_-^m(L_1), \ \s_+^x \s_-^{k-m}(L_2))= (\s_+^{-y} \s_-^n(L_1'), \   \s_+^y \s_-^{l-n}(L_2'))$$ 
or
$$(\s_+^{-x} \s_-^m(L_1), \ \s_+^x \s_-^{k-m}(L_2))= ( \s_+^y \s_-^{l-n}(L_2'), \ \s_+^{-y} \s_-^n(L_1'))$$
for some $m, n \in \Z$ and $x, y \in \Z_{\geq 0}$. The first case is covered in the case when $\K_1 \neq \K_2$.
In the latter case, we obtain by Theorem~\ref{thm-EFM} and Proposition~\ref{prop}, 
$$S^{-x}(T_1) \sim S^y(T_2') \mbox{ and } S^x(T_2) \sim S^{-y}(T_1')$$ 
Since $T_2 \sim T_2'$ we have $S^{x+y}(T_2)  \sim S^{x+y}(T_2')$. Therefore,
$$T_1 \sim S^{-x+x}(T_1) \sim S^{-y+y}(T_1') \sim T_1',$$
which is a contradiction.
\end{proof}

\end{document}